\documentclass[10pt]{amsart}
\usepackage{amssymb, amsmath, amsthm, mathrsfs}
\usepackage{diagrams}
\title{Groupoid cocycles and $K$-theory}
\author{Bram Mesland}
\date\today
\keywords{Groupoid $C^{*}$-algebras, $KK$-theory}
\address{Department of Mathematics, Utrecht University\\Budapestlaan 6\\3584 CD Utrecht\\The Netherlands}
\email{B.Mesland@uu.nl}
\theoremstyle{plain}

\newtheorem{theorem}{Theorem}[subsection]
\newtheorem{corollary}[theorem]{Corollary}

\newtheorem{proposition}[theorem]{Proposition}

\newtheorem{lemma}[theorem]{Lemma}

\theoremstyle{definition}

\newtheorem{definition}[theorem]{Definition}

\DeclareFontFamily{OT1}{pzc}{}
\DeclareFontShape{OT1}{pzc}{m}{it}{<-> s * [1.20] pzcmi7t}{}
\DeclareMathAlphabet{\mathpzc}{OT1}{pzc}{m}{it}

\newcommand{\R}{\mathbb{R}}

\newcommand{\Z}{\mathbb{Z}}

\newcommand{\N}{\mathbb{N}}
\newcommand{\K}{\mathbb{K}}
\newcommand{\C}{\mathbb{C}}

\newcommand{\Aut}{\textnormal{Aut}}

\newcommand{\im}{\mathfrak{Im}}

\newcommand{\id}{\textnormal{id}}

\newcommand{\Dom}{\mathfrak{Dom}}

\newcommand{\tildeotimes}{\tilde{\otimes}}

\newcommand{\lrh}{\leftrightharpoons}

\newcommand{\End}{\textnormal{End}}
\newcommand{\Endst}{\textnormal{End}^{*}}
\newcommand{\supp}{\textnormal{supp }}

\newcommand{\Fin}{\textnormal{Fin}}

\newcommand{\Ind}{\textnormal{Ind}}

\begin{document}

\begin{abstract}
Let $c:\mathcal{G}\rightarrow\R$ be a cocycle on a locally compact Hausdorff groupoid $\mathcal{G}$ with Haar system,  and $\mathcal{H}$ the subgroupoid $\ker c\subset \mathcal{G}$. Under some mild conditons (satisfied by e.g. all integral cocycles on an \'{e}tale groupoid), $c$ gives rise to an unbounded odd $\R$-equivariant bimodule $(\mathpzc{E},D)$ for the pair of $C^{*}$-algebras $(C^{*}(\mathcal{G}),C^{*}(\mathcal{H}))$. If the cocycle comes from a continuous quasi-invariant measure on the unit space $\mathcal{G}^{(0)}$, the corresponding element $[(\mathpzc{E},D)]$ in $KK_{1}(C^{*}(\mathcal{G}),C^{*}(\mathcal{H}))$  gives rise to an index map $K_{1}(C^{*}(\mathcal{G}))\rightarrow \C$.\newline\newline
Keywords: Groupoid $C^{*}$-algebras; $KK$-theory.
\end{abstract}
\maketitle
\tableofcontents
\section*{Introduction}
Groupoid $C^{*}$-algebras \cite{Re} form a rich class of $C^{*}$-algebras, including group $C^{*}$-algebras, crossed products, graph, Cuntz- and Cuntz-Krieger algebras, and the $C^{*}$-algebras considered in foliation theory. The (now classical) Gel'fand-Naimark theorem tells us that $C^{*}$-algebras can be viewed as noncommutative locally compact Hausdorff topological spaces. Groupoids can be viewed as an intermediate structure, allowing for a topological description of noncommutative $C^{*}$-algebras.\newline\newline
A generalization of manifolds to the realm of operator algebras, is Connes' theory of spectral triples \cite{Con}. These are the unbounded cycles for $K$-homology, and their bivariant version \cite{BJ} can be used to give a description of Kasparov's $KK$-theory \cite{Kas}. An unbounded $KK$-cycle $(\mathpzc{E},D)$ for a pair of separable $C^{*}$-algebras consists of a $C^{*}$-module $\mathpzc{E}$ over $B$, which is also a left module over $A$, and an unbounded regular operator $D$ with compact resolvent (in the $C^{*}$-module sense). Moreover, the subalgebra $\mathcal{A}\subset A$ of elements for which the commutator $[D,a]$ extends to an endomorphism of $\mathpzc{E}$ is required to be dense in $A$.\newline\newline
In this paper we construct odd equivariant unbounded cycles for groupoid $C^{*}$-algebras, coming from a continuous 1-cocycle $c:\mathcal{G}\rightarrow \R$, satisfying some mild technical conditions. In case the groupoid is \'{e}tale, and the image of the cocycle is a discrete subgroup of $\R$, these conditions are automatically satisfied. This provides us with a lot of examples, including all Cuntz-Krieger algebras and crossed products of a commutative algebra by $\Z$. In particular, we obtain the noncommutative torus as a ''fibration" over the circle (section 3.5). \newline\newline
Renault \cite{Re} has shown that continuous 1-cocycles give rise to a 1-parameter group of automorphisms of $C^{*}(\mathcal{G})$. We consider the generator $D$ of this group, viewed as an operator in the $C^{*}$-module completion of $C_{c}(\mathcal{G})$ over $C_{c}(\mathcal{H})$, where the closed subgroupoid $\mathcal{H}=\ker c$. This gives rise to an odd $\R$-equivariant unbounded bimodule $(\mathpzc{E},D)$ for the pair of $C^{*}$-algebras $(C^{*}(\mathcal{G}),C^{*}(\mathcal{H}))$.\newline\newline
The appearance of the auxiliary $C^{*}$-algebra $C^{*}(\mathcal{H})$ might seem undesirable at first, when ones purpose is to study $C^{*}(\mathcal{G})$. In many cases of interest, e.g. when the cocycle comes from a quasi invariant measure on the unit space $\mathcal{G}^{0}\subset\mathcal{G}$, the algebra $C^{*}(\mathcal{H})$ carries a canonical trace $\tau$, inducing a homomorphism \[\tau_{*}:K_{0}(C^{*}(\mathcal{H}))\rightarrow\C.\]
Composing this with the homomorphism
\[K_{1}(C^{*}(\mathcal{G}))\xrightarrow{\otimes_{[D]}} K_{0}(C^{*}(\mathcal{H})),\]
coming from the Kasparov product with the above mentioned cycle, yields and index map $K_{1}(C^{*}(\mathcal{G}))\rightarrow\C$.\newline\newline
We will also relate our work to a construction of $KK$-cycles from circle actions in \cite{CNNR}, by showing that the technical condition appearing there is automatically satisfied when the circle action comes from a cocycle. \newline\newline
In order to make the paper readable for the non-specialist, the first part of the paper is devoted to a review of the necessary concepts concerning unbounded selfadjoint operators on $C^{*}$-modules, as well as those from the theory of locally compact Hausdorff groupoids, their actions, and their $C^{*}$-algebras.
\section*{Acknowledgements}
This research was done as part of my Ph.D. thesis during my stay at the Max Planck Institute for Mathematics in Bonn, Germany. It was finished at Utrecht University, the Netherlands. I am thankful to these institutions for their financial support.

I am indebted to Alan Carey, Eli Hawkins, Klaas Landsman, Matilde Marcolli, Ryszard Nest, Jean Renault and Georges Skandalis for useful conversations and/or correspondence. I thank Walter van Suijlekom for a detailed discussion of the example in section 3.5.
\section{$KK$-theory via unbounded operators}
Kasparovs $KK$-theory \cite{Kas} is a central tool in operator $K$-theory. It associates a $\Z/2$-graded group $KK_{*}(A,B)$ to any pair of separable $C^{*}$-algebras $(A,B)$. 
\subsection{Equivariant $C^{*}$-modules}
The central objects one deals with in $KK$-theory are $C^{*}$-modules. A basic reference for these objects is \cite{Lan}. Recall that the \emph{strict topology} on the endomorphisms $\End (V)$ of a normed linear space $V$ is given by pointwise norm convergence, i.e. $\phi_{n}\rightarrow \phi$ in $\End(V)$ if and only if \[\|\phi_{n}(v)-\phi(v)\|\rightarrow 0\] for all $v\in V$. We will always consider $\Aut(B)$, the automorphisms of a $C^{*}$-algebra $B$, with this topology.
\begin{definition}\label{C*mod} Let $G$ be a second countable locally compact group. A $C^{*}$-algebra $B$ is a $G$-\emph{algebra} if there is a continuous homomorphism $G\rightarrow \Aut(B)$. A \emph{right} $C^{*}$-$B$-\emph{module} is a complex vector space $\mathpzc{E}$ which is also a right $B$-module, and a bilinear pairing
\[\begin{split}\mathpzc{E}\times\mathpzc{E} &\rightarrow B \\
(e_{1},e_{2}) &\mapsto \langle e_{1},e_{2}\rangle, \end{split}\]
such that
\begin{itemize}\item $\langle e_{1},e_{2}\rangle=\langle e_{2},e_{1}\rangle^{*},$
\item $\langle e_{1},e_{2}b\rangle=\langle e_{1},e_{2}\rangle b,$ 
\item $\langle e,e\rangle\geq 0$ and $\langle e,e\rangle=0\Leftrightarrow e=0,$\item $\mathpzc{E}$ is complete in the norm $\|e\|^{2}:=\|\langle e,e\rangle\|.$\end{itemize}
If $B$ is a $G$-algebra, such $\mathpzc{E}$ is a $G$-\emph{module} if it comes equipped with a strictly continuous $G$-action satisfying \begin{itemize}\item $g(eb)=(ge)gb$ \item $\langle ge_{1},ge_{2}\rangle=g\langle e_{1},e_{2}\rangle$.\end{itemize}
We use the notation  $\mathpzc{E}\leftrightharpoons B$ to indicate this structure. 
\end{definition}
It turns out that a $B$-linear endomorphism of a $C^{*}$-module $\mathpzc{E}\leftrightharpoons B$ does not always admit an adjoint. However, requiring the existence of an adjoint is enough to obtain a number of other desirable properties. Let \[\Endst_{B}(\mathpzc{E}):=\{T:\mathpzc{E}\rightarrow\mathpzc{E}:\quad\exists T^{*}:\mathpzc{E}\rightarrow\mathpzc{E}, \quad\langle Te_{1},e_{2}\rangle=\langle e_{1},T^{*}e_{2}\rangle\}.\]
Elements of $\Endst_{B}(\mathpzc{E})$ are called \emph{adjointable
operators}.
\begin{proposition}Let $T\in\Endst_{B}(\mathpzc{E})$. Then $T$ is linear, bounded, and a $B$-module morphism. Moreover, $\Endst_{B}(\mathpzc{E})$ is a $C^{*}$-algebra in the operator norm 
and the involution $T\mapsto T^{*}$.\end{proposition}
If $\mathpzc{E}$ is a $G$-module, there is a strictly continuous $G$-action on $\Endst_{B}(\mathpzc{E})$, given by $g:T\mapsto gTg^{-1}$. That is, $\Endst_{B}(\mathpzc{E})$ is a $G$-algebra. If $A$ is another $C^{*}$-algebra, an $(A,B)$-\emph{bimodule} is a $C^{*}$-$B$-module $\mathpzc{E}$ together with a *-homomorphism $A\rightarrow\Endst_{B}(\mathpzc{E})$. This structure is denoted
\[A\rightarrow\mathpzc{E}\lrh B.\]
 Such a bimodule is called $\emph{equivariant}$ if it is a $G$-module, $A$ is a $G$-algebra and $a(ge)=(ga)e$. It is called a \emph{correspondence} if the $A$ representation is \emph{essential}, that is, if $A\mathpzc{E}$ is dense in $\mathpzc{E}$.\newline

Note
that the involution on $B$ allows for considering $\mathpzc{E}$ as a left
$B$-module via $be:=eb^{*}$. The inner product can be used to turn the algebraic tensor product
$\mathpzc{E}\otimes_{B}\mathpzc{E}$ into a $*$-algebra:
\[ e_{1}\otimes e_{2}\circ f_{1}\otimes f_{2}:=e_{1}\langle
e_{2},f_{1}\rangle\otimes f_{2},\quad (e_{1}\otimes e_{2})^{*}:=e_{2}\otimes
e_{1}.\] This algebra is denoted by $\Fin_{B}(\mathpzc{E}).$
There is an injective *-homomorphism
\[\Fin_{B}(\mathpzc{E})\rightarrow\Endst_{B}(\mathpzc{E}),
\]
given by $e_{1}\otimes e_{2}(e):=e_{1}\langle e_{2},e\rangle$. The closure of
$\Fin_{B}(\mathpzc{E})$ in the operator norm is the $C^{*}$-algebra of
$B$-\emph{compact operators} on $\mathpzc{E}$. It is denoted by
$\mathbb{K}_{B}(\mathpzc{E})$. If $\mathpzc{E}$ is a $G$-module, then the $G$-action that $\K_{B}(\mathpzc{E})$ inherits from $\Endst_{B}(\mathpzc{E})$ is norm continuous. 
Two $C^{*}$ algebras $A$ and $B$ are said to be \emph{strongly Morita equivalent} if there exists a $C^{*}$-module $\mathpzc{E}\leftrightharpoons B$ such that $A\cong\mathbb{K}_{B}(\mathpzc{E})$. The bimodule $A\rightarrow\mathpzc{E}\lrh B$ is called a \emph{Morita equivalence bimodule}. A Morita equivalence bimodule is in particular a correspondence.
\subsection{Unbounded regular operators}
Similar to the Hilbert space setting, there is a notion of unbounded operator on
a $C^{*}$-module.We refer to \cite{Baaj}, \cite{Lan} and \cite{Wor} for detailed
expositions of this theory.
\begin{definition}[\cite{BJ}] Let $\mathpzc{E}$ be a
$C^{*}$-$B$-module. A densely defined closed operator
$D:\mathfrak{Dom}D\rightarrow\mathpzc{E}$ is called
\emph{regular} if 
\begin{itemize} \item $D^{*}$ is densely defined in $\mathpzc{E}$;
\item $1+D^{*}D$ has dense range.
\end{itemize}\end{definition}
Such an operator is automatically $B$-linear, and $\mathfrak{Dom}D$ is a $B$-submodule of $\mathpzc{E}$. 
There are two operators,
$\mathfrak{r}(D),\mathfrak{b}(D)\in\Endst_{B}(\mathpzc{E})$
canonically associated with a regular operator $D$. They are the \emph{resolvent} of $D$
\begin{equation} \mathfrak{r}(D):=(1+D^{*}D)^{-\frac{1}{2}},\end{equation}
and the \emph{bounded transform}
\begin{equation}\mathfrak{b}(D):=D(1+D^{*}D)^{-\frac{1}{2}}.\end{equation}
To construct selfadjoint regular operators in practice, we include some remarks and results on the extension of symmetric regular operators. 
A densely defined operator $D$ in a
$C^{*}$-module $\mathpzc{E}$ is \emph{symmetric} if, for $e,e'\in\mathfrak{Dom}D$ we have $\langle De,e'\rangle=\langle e,De'\rangle$. Symmetric operators are closable, and their closure is again symmetric. Hence we will
tacitly assume all symmetric operators to be closed. 
\begin{lemma}[\cite{Lan}]\label{closedrange} Let $D$ be a densely defined symmetric operator. Then the operators $D+i$ and $D-i$ are injective and have closed range. 
\end{lemma} We can now define two isometries
\[\mathfrak{u}_{+}(D):=(D+i)\mathfrak{r}(D),\quad \mathfrak{u}_{-}(D):=(D-i)\mathfrak{r}(D),\]
and the \emph{Cayley transform} of $D$ is
\begin{equation}\label{Cayley}\mathfrak{c}(D):=\mathfrak{u}_{-}(D)\mathfrak{u}_{+}(D)^{*}.\end{equation}
In general, $\mathfrak{c}(D)$ is a partial isometry, with closed range. The operator $D$ can be recoverd from $\mathfrak{c}(D)$ by the formulas
\[\begin{split}\mathfrak{Dom}(D)&=\mathfrak{Im}(1-\mathfrak{c}(D))\mathfrak{c}(D)^{*}\\
D(1-\mathfrak{c}(D))\mathfrak{c}(D)^{*}e &=i(1+\mathfrak{c}(D))\mathfrak{c}^{*}(D)e.\end{split}\]
\begin{theorem}[\cite{Lan}]The Cayley transform $\mathfrak{c}$ furnishes a bijection between the set of symmetric regular operators in $\mathpzc{E}$ and the set of partial isometries $c\in\End_{B}^{*}(\mathpzc{E})$ with
the property that $(1-c)c^{*}$ has dense range. Moreover, $D'$ is an extension of $D$ if and and only if $\mathfrak{c}(D')$ is an extension of $\mathfrak{c}(D)$.
\end{theorem}
For a selfadjoint regular operator $D$, $1+D^{2}$ has dense range.Therefore by lemma \ref{closedrange}, the operators $D+i$ and $D-i$ are bijective.
\begin{corollary}\label{symself} A symmetric regular operator $D$  is selfadjoint if and only if $\mathfrak{c}(D)$ is unitary if and only if $D+i$ and $D-i$ have dense range.
\end{corollary}
\subsection{$KK$-theory}
The theory was originally described by Kasparov \cite{Kas} using adjointable (and hence bounded) endomorphisms of $C^{*}$-modules. Later Connes \cite{Con} and Baaj-Julg \cite{BJ} gave a description of the cycles of K-homology and $KK$-theory using unbounded endomorphisms. It is this formulation that we will employ here.
The main references for the conventional (i.e. bounded) approach to $KK$-theory are Kasparov's original papers \cite{Kas, Kas2}. Given a locally compact group $G$, a $\Z/2$-graded bifunctor $KK^{G}_{*}$ is constructed, on pairs $(A,B)$ of $G$-algebras. This bifunctor has remarkable properties.  Amongst them is the following deep theorem of Kasparov.
\begin{theorem}[\cite{Kas, Kas2}] For any $C^{*}$-$G$-algebras $A,B,C$ there exists an associative bilinear pairing
\[KK_{i}^{G}(A,B)\otimes_{\Z} KK_{j}^{G}(B,C)\xrightarrow{\otimes_{B}} KK_{i+j}^{G}(A,C).\] 
\end{theorem}
In particular, the Kasparov product with an element in $KK_{i}^{G}(A,B)$ yields homomorphisms
\[K_{*}^{G}(A)\rightarrow K_{*+i}^{G}(B).\] 
In this paper we will be concerned with odd $KK$-theory. The cycles for the group $KK_{1}^{G}(A,B)$ are given in the following definition.
\begin{definition} Let $A,B$ be $G$-algebras. An \emph{odd equivariant unbounded bimodule} $(\mathpzc{E},D)$ for $(A,B)$ is given by an equivariant $C^{*}$-bimodule  $A\rightarrow\mathpzc{E}\leftrightharpoons B$ together with an unbounded regular operator $D$ in $\mathpzc{E}$ such that
\begin{itemize}\item $[D,a]\in\Endst_{B}(\mathpzc{E})$ for all $a$ in some dense subalgebra of $A$, \item $a\mathfrak{r}(D)\in\K_{B}(\mathpzc{E})$,\item The map $g\mapsto D- gDg^{-1}$ is a strictly continuous map $G\rightarrow \Endst_{B}(\mathpzc{E})$.\end{itemize}
\end{definition}
\section{Groupoids}
We now recall the theory of groupoids, their $C^{*}$-algebras and correspondences. Although we will only encounter the particularly simple type of correspondence given by a closed subgroupoid, we will need the general results on the bimodules and Morita equivalences they induce.
\subsection{Haar systems and the convolution algebra}
In general, topological groupoids can be viewed as generalizations of both groups and topological spaces. Both of these occur as extreme cases of the following definition.
\begin{definition} A \emph{groupoid} $\mathcal{G}$ is a small category in which
every morphism is invertible. The set of morphisms of $\mathcal{G}$ is denoted $\mathcal{G}^{(1)}$, and the objects $\mathcal{G}^{(0)}$.
We identify $\mathcal{G}^{(0)}$ with a subset of $\mathcal{G}^{(1)}$ as identity morphisms.
The groupoid $\mathcal{G}$ is said to be a \emph{locally compact Hausdorff} if $\mathcal{G}^{(1)}$ carries such a topology, and the domain and range maps
\[d,r:\mathcal{G}^{(1)}\rightarrow \mathcal{G}^{(0)}\subset \mathcal{G}^{(1)},\]
are continuous for this topology. It is said to be \emph{\'{e}tale} if $r$ and $d$ are local homeomorphisms.
\end{definition}
Thus, a group can be regarded as a groupoid with just one object, and a topological space as a groupoid with only identity morphisms. We will tacitly assume all groupoids to be locally compact and Hausdorff.

We will consider groupoids with the additional datum of a Haar system. This is a system of measures supported on the fibers of the range map $r$. Inversion in the groupoid yields a system of
measures supported on the fibers of $d$.
\begin{definition}[\cite{Re}] Let $\mathcal{G}$ be a locally compact Hausdorff groupoid. A \emph{Haar system} on $\mathcal{G}$ is a system of measures
$\{\nu^{x}:x\in\mathcal{G}^{(0)}\}$
on $\mathcal{G}^{(1)}$ such that
\begin{itemize}\item $\textnormal{supp }\nu^{x}=r^{-1}(x)$
\item $\forall f\in C_{c}(\mathcal{G}),\quad \int_{\mathcal{G}}f(\xi)d\nu^{r(\eta)}(\xi)=\int_{\mathcal{G}}f(\eta\xi)d\nu^{d(\eta)}(\xi)$
\item $\forall f\in C_{c}(\mathcal{G}),\quad g(x):=\int_{\mathcal{G}}f(\xi)d\nu^{x}(\xi)\in C(\mathcal{G}^{(0)}).$\end{itemize}\end{definition}
\'{E}tale groupoids always admit a Haar system, consisting of counting measures on the fibers.
There is a natural involution on $C_{c}(\mathcal{G})$ given by $f^{*}(\xi):=\overline{f(\xi^{-1})}$.
The Haar system also allows us to define the \emph{convolution product} in $C_{c}(\mathcal{G})$:
\[f*g(\eta):=\int_{\mathcal{G}}f(\xi)g(\xi^{-1}\eta)d\nu^{r(\eta)}.\]
This is an associative, distributive product that makes $C_{c}(\mathcal{G})$ into a topological *-algebra for the topology given by uniform convergence on compact subsets.
\begin{definition}[\cite{Hahn}]\label{full} Let $\mathcal{G}$ be a locally compact Hausdorff groupoid with Haar system. Define
\[\|f\|_{\nu}:=\sup_{u\in\mathcal{G}^{(0)}}\int_{\mathcal{G}}|f(\xi)|d\nu^{u},\quad\|f\|_{\nu^{-1}}:=\sup_{u\in\mathcal{G}^{(0)}}\int_{\mathcal{G}}|f(\xi^{-1})|d\nu^{u},\]
and
\[\|f\|_{I}:=\max\{\|f\|_{\nu},\|f\|_{\nu^{-1}}\}.\]
\end{definition}
Let $\mathcal{H}$ be a Hilbert space. A representation $\pi:C_{c}(\mathcal{G})\rightarrow B(\mathcal{H})$ is called \emph{admissible} if it is continuous with respect to the inductive limit topology on $C_{c}(\mathcal{G})$ and the weak operator topology on $B(\mathcal{H})$, and $\|\pi(f)\|\leq\|f\|_{I}$.
\begin{definition}[\cite{Re}]The \emph{full $C^{*}$-norm} on $C_{c}(\mathcal{G})$ is defined by
\[\|f\|:=\sup\{\|\pi(f)\|:\pi\textnormal{ admissible}\}.\]
The \emph{full $C^{*}$-algebra} $C^{*}(\mathcal{G})$ is the completion of $C_{c}(\mathcal{G})$ with respect to this norm. 
\end{definition}
The space $C_{c}(\mathcal{G})$ is a right module over $C_{c}(\mathcal{G}^{(0)})$ if we define
\[f*g(\xi):=f(\xi)g(d(\xi)),\quad f\in C_{c}(\mathcal{G}),\quad g\in C_{c}(\mathcal{G}^{(0)}).\]We can associate a canonical $C^{*}$-$C_{0}(\mathcal{G}^{(0)})$-module to a groupoid with Haar system  
via the pairing
\[\begin{split}C_{c}(\mathcal{G})\times
C_{c}(\mathcal{G}) &\rightarrow C_{c}(\mathcal{G}^{(0)})\\ \langle f,h\rangle(u)\quad&:=\int_{\mathcal{G}}\overline{f(\xi^{-1})}h(\xi^{-1})d\nu^{u}\xi.\end{split}\]
As usual, $C_{c}(\mathcal{G})$ gets a norm 
\[\|f\|^{2}:=\|\langle f,f\rangle\|:=\sup_{u\in\mathcal{G}^{(0)}} \int_{\mathcal{G}}|f(\xi^{-1})|^{2}d\nu^{u}\xi.\]
We denote the completion  of $C_{c}(\mathcal{G})$ in this norm by $L^{2}(\mathcal{G},\nu)$.
Since $C_{c}(\mathcal{G})$ acts on itself by convolution we get an embedding
\[C_{c}(\mathcal{G})\hookrightarrow \Endst_{C(\mathcal{G}^{(0)})}(L^2(\mathcal{G},\nu)).\] 
\begin{definition}[\cite{Re}]\label{red}The \emph{reduced $C^{*}$-algebra} 
$C^{*}_{r}(\mathcal{G})$ of $\mathcal{G}$, is the completion of $C_{c}(\mathcal{G})$ in the norm $\|.\|_{r}$ it gets as an algebra of operators on $L^{2}(\mathcal{G},\nu).$\end{definition}
This approach to defining $C^{*}_{r}(\mathcal{G})$ is different from that in \cite{Re} and was first considered in \cite{Skan}. As mentioned before, the $C^{*}$-algebras $C^{*}(\mathcal{G})$ and
$C^{*}_{r}(\mathcal{G})$ are not isomorphic in general. A sufficient condition for the algebras to coincide is that of \emph{amenability} \cite{An}.

\subsection{Groupoid actions}
 If $\phi_{i}:X_{i}\rightarrow Y$, $i=1,2$, are continuous maps between topological 
spaces $X_{i}$ and $Y$, we denote the pull back, or \emph{fibered product}, of the $X_{i}$ over $Y$ by
\[X_{1}*_{Y}X_{2}:=\{(x_{1},x_{2}):\phi_{1}(x_{1})=\phi_{2}(x_{2})\}.\]  The space $X_{1}*_{Y}X_{2}$ is the universal solution for commutative diagrams 
\begin{diagram} X &\rTo^{\psi_{1}} & X_{1}\\ \dTo^{\psi_{2}} & &\dTo^{\phi_{1}}\\X_{2} &\rTo^{\phi_{2}} & Y.\end{diagram} 
In case one of the $X_{i}$ is a groupoid $\mathcal{G}$ and a map $\rho:Z\rightarrow\mathcal{G}^{(0)}$ is given, it is convenient to write $\mathcal{G}\ltimes_{\rho} Z$ for the pull back with respect to $d$ and $\rho$, 
and $Z\rtimes_{\rho}\mathcal{G}$ for the pull back with respect to $r$ and $\rho$.
\begin{definition}[cf. \cite{Klaas1},\cite{Mcrun}] Let $Z$ be a topological space and $\mathcal{G}$ a groupoid. A \emph{left action} of $\mathcal{G}$ on $Z$ consists of a continuous map $\rho:Z\rightarrow \mathcal{G}^{(0)}$, called the \emph{moment map}, and a continuous map 
\[\begin{split}\mathcal{G}\ltimes_{\rho}Z & \rightarrow Z \\ (\xi,z) &\mapsto \xi z,\end{split}\] (the pull back is with respect to $d:\mathcal{G}\rightarrow \mathcal{G}^{(0)}$) with following properties:
\begin{itemize} \item $\rho (\xi z)=r(\xi),$ \item $\rho (z) z=z,$ \item If $(\xi_{1},\xi_{2})\in\mathcal{G}^{2}$ and $(\xi_{2},z)\in\mathcal{G}\ltimes_{\rho}Z  \rightarrow Z$ then $(\xi_{1}\xi_{2}) z=\xi_{1} (\xi_{2} z).$\end{itemize} The space $Z$ is said to be a \emph{left $\mathcal{G}$-bundle}.
\end{definition}
 
The notion of right action is obtained by switching $r$ and $d$ and considering  $Z\rtimes_{\rho}\mathcal{G}$. The spaces $Z\rtimes_{\rho}\mathcal{G}$ and $\mathcal{G}\ltimes_{\rho} Z$ are groupoids over 
$Z$. We will describe the structure for $Z\rtimes_{\rho}\mathcal{G}$. The structure for $\mathcal{G}\ltimes_{\rho} Z$ is similar. We have
\[Z\rtimes_{\rho}\mathcal{G}=\{(z,\xi)\in Z\times\mathcal{G}:\rho(z)=r(\xi)\},\]
and define
\[d(z,\xi):=z\xi,\quad r(z,\xi)=z,\quad (z,\xi)^{-1}=(z\xi,\xi^{-1}),\quad (z,\xi)(z\xi,\eta)=(z,\xi\eta).\]
This is well defined because $Z$ is a $\mathcal{G}$-bundle. If $Z$ carries both a left $\mathcal{G}$- and a right $\mathcal{H}$-action the actions are said to \emph{commute} if 
\begin{itemize}\item $\forall (\xi,z)\in\mathcal{G}\ltimes_{\rho}Z,(z,\chi)\in Z\rtimes_{\sigma}\mathcal{H},\quad(\xi z)\chi=\xi(z\chi),$ 
\item $\forall (z,\chi)\in Z\rtimes_{\sigma}\mathcal{H},\quad\rho(z\chi)=\rho(z),$
\item $\forall (\xi,z)\in \mathcal{G}\ltimes_{\rho}Z,\quad \sigma (\xi z)=\sigma(z)$. \end{itemize}  
Such a $Z$ is called a $\mathcal{G}$-$\mathcal{H}$-\emph{bibundle}. Moreover, the action is said to be \emph{left proper} if the map 
\[\begin{split} \mathcal{G}\ltimes_{\rho}Z &\rightarrow Z\times Z \\
(\xi,z) &\mapsto (\xi z, z),\end{split}\] is proper, that is, inverse images of compact sets are compact. Right properness is defined similarly. The notions of Morita equivalence and correspondence for groupoids are defined in terms of bibundles equipped with extra structure.

\begin{definition}[cf.\cite{Klaas1},\cite{Mcrun},\cite{SO}] Let $Z$ be a $\mathcal{G}$-$\mathcal{H}$ bibundle with moment maps $\rho:Z\rightarrow \mathcal{G}^{(0)}$ and $\sigma:Z\rightarrow \mathcal{H}^{(0)}$. The $\mathcal{G}$ action is said to be 
\emph{left principal} if the map
\[\begin{split} \mathcal{G}\ltimes_{\rho}Z &\rightarrow Z*_{\mathcal{H}^{(0)}}Z \\
(\xi,z) &\mapsto (\xi z, z),\end{split}\] is a homeomorphism. This is equivalent to saying that the $\mathcal{G}$-action is free, $\sigma$ is an open surjection and induces a bijection $\mathcal{G}\backslash
Z\rightarrow\mathcal{H}^{(0)}$. A bibundle which is both left principal and right proper is said to be a \emph{correspondence}, and is denoted
\[\mathcal{G}\rightarrow Z\lrh\mathcal{H}.\]
If the bibundle is both left- and right-principal, it is said to be an \emph{equivalence bibundle}. Two groupoids $\mathcal{G},\mathcal{H}$ are \emph{Morita equivalent} 
if there exists an equivalence $\mathcal{G}$-$\mathcal{H}$-bibundle.\end{definition}
Groupoid correspondences provide one with a well behaved notion of morphism for groupoids, first observed in \cite{HilSkan} and later developed in \cite{Klaas1,Klaas2} and \cite{SO}. In the definition of correspondence of \cite{SO}, the moment map $\rho$, like $\sigma$, is assumed to be surjective. This condition is absent in \cite{Klaas1}, and is not needed for the construction of the bimodules in the next section, as noted in \cite{Klaas1}.
 
\subsection{$C^{*}$-modules from correspondences}
Groupoid correspondences and equivalences give rise to correspondences and Morita equivalences for the full and reduced $C^{*}$-algebras. In the theorem below, it must be mentioned that the result concerning the Morita equivalence of the reduced $C^{*}$-algebras seems to be well known, and has been stated without proof many times in the literature. A proper proof of this result has been written down recently in \cite{SW}.
\begin{theorem}[\cite{Klaas1},\cite{MRW},\cite{SW},\cite{SO}]\label{bimod} Let $\mathcal{G}$ and $\mathcal{H}$ be groupoids with Haar system, and $\mathcal{G}\rightarrow Z\leftrightharpoons\mathcal{H}$ a groupoid correspondence. The space $C_{c}(Z)$ can be completed into $C^{*}$-correpondences 
\[C^{*}(\mathcal{G})\rightarrow\mathpzc{E}^{Z}\leftrightharpoons C^{*}(\mathcal{H})\quad\textnormal{and}\quad C^{*}_{r}(\mathcal{G})\rightarrow\mathpzc{E}^{Z}_{r}\leftrightharpoons C^{*}_{r}(\mathcal{H}).\] 
When the correspondence $\mathcal{G}\rightarrow Z\leftrightharpoons\mathcal{H}$ is an equivalence bibundle, the above $C^{*}$-correspondences are Morita equivalence bimodules.
\end{theorem}
On the dense subspaces $C_{c}(\mathcal{G})$, $C_{c}(\mathcal{H})$ and $C_{c}(Z)$, explicit formulae for both the inner product(s) and module structures can be given. For later reference and completeness we give them here. 
For $\Phi\in C_{c}(Z)$, the right module action of $h\in C_{c}(\mathcal{H})$ is given by
\begin{equation}\label{modH}\Phi\cdot h (z):=\int_{\mathcal{H}}\Phi(z\chi)h(\chi^{-1})d\nu^{\sigma(z)}\chi.\end{equation}
Similarly, the left action of $g\in C_{c}(\mathcal{G})$ on $\Phi$ is
\begin{equation}\label{modG}g\cdot\Phi(z):=\int_{\mathcal{G}}g(\xi)\Phi(\xi^{-1}z)d\nu^{\rho(z)}\xi.\end{equation}
There is a $C_{c}(\mathcal{H})$-valued inner product on $C_{c}(Z)$:
\begin{equation}\label{innH}\langle \Phi,\Psi\rangle_{\mathcal{H}}(\chi):=\int_{\mathcal{G}}\overline{\Phi(\xi^{-1}z)}\Psi(\xi^{-1}z\chi)dv^{\rho(z)}\xi.\end{equation}
In this formula, $z\in Z$ is chosen such that $\sigma(z)=r(\chi)$, and it is independent of choice because $\mathcal{G}\setminus Z\cong\mathcal{H}^{(0)}$, and finite because the $\mathcal{G}$-action is proper. We have $\langle \Phi,\Psi\rangle_{\mathcal{H}}\in C_{c}(\mathcal{H})$ by virtue of the properness of the $\mathcal{H}$-action. In case the $\mathcal{H}$ action is transitive, one defines a $C_{c}(\mathcal{G})$-valued inner product by
\begin{equation}\label{innG}\langle \Phi,\Psi\rangle_{\mathcal{G}}(\eta):=\int_{\mathcal{H}}\Phi(\eta^{-1}z\chi)\overline{\Psi(z\chi)}d\nu^{\sigma(z)}\chi,\end{equation}
where $z\in Z$ is chosen in such a way that $\rho(z)=r(\eta)$. Again, the integral is independent of this choice by transitivity of the $\mathcal{H}$-action.
\section{Cocycles and $K$-theory}
The continuous cohomology of a groupoid generalizes that of a group. In this section we develop a connection between the cocycles defining the cohomology group $H^{1}(\mathcal{G},\R)$ and $K_{1}(C^{*}(\mathcal{G}))$. This is done by constructing
for each exact real-valued 1-cocycle $c:\mathcal{G}\rightarrow\R$ an odd unbounded $(C^{*}(\mathcal{G}),C^{*}(\mathcal{H}))$-bimodule, where $\mathcal{H}=\ker c$. This in turn induces maps
$K_{1}(C^{*}(\mathcal{G}))\rightarrow K_{0}(C^{*}(\mathcal{H}))$ and $K_{0}(C^{*}(\mathcal{G}))\rightarrow K_{1}(C^{*}(\mathcal{H}))$. According to properties of $c$, the $K$-groups of $C^{*}(\mathcal{H})$ can be more accessible than those of
$C^{*}(\mathcal{G})$, thus paving a way to the calculation of invariants of $C^{*}(\mathcal{G})$.
\subsection{Groupoid cocycles}
The cohomology of groupoids can be developed in complete generality, by adapting the theory for groups, in a similar way as the notion of action is adapted. A detailed decription of groupoid cohomology can be found in \cite{Re}. We will only be interested in continuous 1-cocycles satisfying some regularity property. 
\begin{definition} Denote by $Z^{1}(\mathcal{G},\R)$ the set of continuous homomorphisms $\mathcal{G}\rightarrow\R$. We will refer to the elements of $Z^{1}(\mathcal{G},\R)$ as \emph{cocycles} on $\mathcal{G}$. Denote by
$B^{1}(\mathcal{G},\R)$ the subset of those $c\in Z^{1}(\mathcal{G},\R)$ such that there exists a continuous function $f:\mathcal{G}^{(0)}\rightarrow\R$ such that
$c(\xi)=f(r(\xi))-f(d(\xi))$. The elements of $B^{1}(\mathcal{G},\R)$ are referred to as \emph{coboundaries}. 
\end{definition}
One defines $H^{1}(\mathcal{G},\R):=Z^{1}(\mathcal{G},\R)/B^{1}(\mathcal{G},\R)$, as usual, but we will not use this group in the present paper.
The kernel 
\[\ker c:=\{\xi\in\mathcal{G}:c(\xi)=0\}\]
of a continuous cocycle is a closed subgroupoid of $\mathcal{G}$, which we will denote by $\mathcal{H}$. It is immediate that $\mathcal{H}^{(0)}=\mathcal{G}^{(0)}$. $\mathcal{H}$ acts on $\mathcal{G}$ by both left- and right multiplication, and
these actions are proper. We will always consider the action by multiplication from the right. The resulting bibundle $\mathcal{G}\rightarrow\mathcal{G}\leftrightharpoons\mathcal{H}$ is a correspondence.

Any closed subgroupoid with Haar
system
$\mathcal{H}\subset\mathcal{G}$ is Morita equivalent to the crossed product $\mathcal{G}\ltimes_{r}\mathcal{G}/\mathcal{H},$ 
where the moment map $\mathcal{G}/\mathcal{H}\rightarrow\mathcal{G}^{(0)}$ for the action of $\mathcal{G}$ on $\mathcal{G}/\mathcal{H}$ is given by $[\chi]\mapsto r(\chi)$, whence the notation. 
The groupoid $\mathcal{G}\ltimes_{r}\mathcal{G}/\mathcal{H}$ inherits a Haar system from $\mathcal{G}$, 
since we have
\[r^{-1}([\eta])=\{(\xi,[\eta])\in\mathcal{G}\ltimes_{r}\mathcal{G}/\mathcal{H}:d(\xi)=r(\eta)\}\cong d^{-1}(r(\eta)).\] 
The equivalence correspondence is given by $\mathcal{G}$ itself with moment map \[\begin{split}\rho:\mathcal{G}&\rightarrow (\mathcal{G}\ltimes_{r}\mathcal{G}/\mathcal{H})^{(0)}=\mathcal{G}/\mathcal{H}\\\eta &\mapsto [\eta]\end{split}\]
equal to the quotient map. The left action is given by 
\begin{equation}\label{left}(\xi,[\eta_{1}])\eta_{2}=\xi\eta_{2},\end{equation} 
whenever $[\eta_{1}]=[\eta_{2}]$, and hence the bundle is left principal. The map $\sigma:\mathcal{G}\rightarrow\mathcal{H}^{(0)}$ is just equal to $d$. The bundle is right principal by construction. \newline

Recall that a map $\phi:X\rightarrow Y$ between topological spaces is a \emph{quotient map} if a subset $U\subset Y$ is open if and only if $\phi^{-1}(U)$ is open in $X$. That is, $Y$ carries the quotient topology defined by $\phi$.
\begin{definition}\label{exact}A cocycle $c:\mathcal{G}\rightarrow\R$ is \emph{regular} if $\mathcal{H}=\ker c$ admits a Haar system, and \emph{exact} if it is regular and the map 
\[\begin{split}r\times c:\mathcal{G}&\rightarrow\mathcal{G}^{(0)}\times \R \\ 
\xi &\mapsto (r(\xi),c(\xi))\end{split}\] 
is a quotient map onto its image.
\end{definition}
From the above discussion, it follows that for a regular cocycle,  the groupoid correspondence $\mathcal{G}\rightarrow\mathcal{G}\leftrightharpoons\mathcal{H}$ induces a correspondence $C^{*}(\mathcal{G})\rightarrow \mathpzc{E}^{\mathcal{G}}\lrh C^{*}(\mathcal{H})$, via theorem \ref{bimod}. For the reduced $C^{*}$-algebras, we get a correspondence $C^{*}_{r}(\mathcal{G})\rightarrow \mathpzc{E}^{\mathcal{G}}_{r}\lrh C^{*}_{r}(\mathcal{H})$ from the same theorem. Moreover, the full and reduced $C^{*}$-algebras of $\mathcal{G}$ and $\mathcal{G}\ltimes_{r}\mathcal{G}/\mathcal{H}$ are Morita equivalent. If $\mathcal{G}$ is an \'{e}tale groupoid, any closed subgroupoid admits a Haar system, as is the case when $\mathcal{G}$ is a Lie groupoid and $c$ is smooth. \newline
\begin{lemma}\label{quotient} Let $c:\mathcal{G}\rightarrow\R$ be an exact cocycle, and $\mathcal{H}=\ker c$. The map
\[\begin{split}\overline{r\times c}:\mathcal{G}/\mathcal{H}&\rightarrow\mathcal{G}^{(0)}\times\R\\
\xi &\mapsto (r(\xi),c(\xi)),\end{split}\]
is a homeomorphism onto its image. 
\end{lemma}
\begin{proof} First observe that  $\overline{r\times c}$ is a continuous injection: If $r(\xi)=r(\eta)$ and $c(\xi)=c(\eta)$, then $\xi^{-1}\eta\in\mathcal{H}$ and so $[\eta]=[\xi]$ in $\mathcal{G}/\mathcal{H}$.
Moreover, we have $(\overline{r\times c})\circ\rho=r\times c$, which is a quotient map by hypothesis. Since $\mathcal{G}/\mathcal{H}$ carries the quotient topology, the result follows.
\end{proof}
\begin{lemma}\label{openclosed} For a regular cocycle to be exact, it is sufficient that $r\times c$ be either open or closed.
\end{lemma}
\begin{proof}Suppose $r\times c$ is closed. The proof in the open case translates verbatim. $C\subset\mathcal{G}/\mathcal{H}$ is closed if and only if $\rho^{-1}(C)$ is closed in $\mathcal{G}$. Thus
\[(\overline{r\times c})(C)=(\overline{r\times c})\rho\circ\rho^{-1}(C)=(r\times c) \rho^{-1}(C),\]
is closed. Thus $\overline{r\times c}$ is a continuous closed bijection onto its image, and therefore a homeomorphism.
\end{proof}
Renault \cite{Re} showed that a 1-cocycle $c\in Z^{1}(\mathcal{G},\R)$ defines a one-parameter group of automorphisms of $C^{*}(\mathcal{G})$ by
\begin{equation}\label{1par} u_{t}f(\xi)=e^{itc(\xi)}f(\xi).\end{equation}
Furthermore he showed that if $c\in B^{1}(\mathcal{G},\R)$, the automorphism group is inner, i.e. implemented by a strongly continuous family of unitaries in the multiplier algebra of $C^{*}(\mathcal{G})$.  In general, the one-parameter groups of $C^{*}(\mathcal{G})$ and $C^{*}_{r}(\mathcal{G})$ defined by a regular cocycle $c$ can be described conveniently in the bimodules $\mathpzc{E}^{\mathcal{G}}$ and $\mathpzc{E}^{\mathcal{G}}_{r}$.
\begin{proposition}\label{one}Let $c:\mathcal{G}\rightarrow\R$ be a regular cocycle. The operators 
\[\begin{split}U_{t}:C_{c}(\mathcal{G}) &\rightarrow C_{c}(\mathcal{G}) \\
U_{t}f(\xi)&=e^{itc(\xi)}f(\xi)\end{split}\]
extend to a one parameter group of unitaries in $\Endst_{C^{*}(\mathcal{H})}(\mathpzc{E}^{\mathcal{G}})$, resp. $\Endst_{C^{*}_{r}(\mathcal{H})}(\mathpzc{E}^{\mathcal{G}}_{r})$, implementing the one parameter group of 
automorphisms $u_{t}$ of $C^{*}(\mathcal{G})$, resp. $C^{*}_{r}(\mathcal{G}).$\end{proposition}
\begin{proof} The identity $\langle U_{t}f,U_{t}g\rangle_{\mathcal{H}}=\langle f,g\rangle_{\mathcal{H}}$ is proved by a straightforward computation. Since $u_{t}f=f$ for $f\in C_{c}(\mathcal{H})$, $\mathpzc{E}^{\mathcal{G}}$ is an $\R$-module. To see that $U_{t}$ implements $u_{t}$, just compute:
\[\begin{split}U_{t}(f*U_{t}^{*}g)(\eta)&=e^{itc(\eta)}\int_{\mathcal{G}}f(\xi)e^{-itc(\xi^{-1}\eta)}g(\xi^{-1}\eta)d\nu^{r(\eta)}\\
&=\int_{\mathcal{G}}f(\xi)e^{itc(\xi)}g(\xi^{-1}\eta)d\nu^{r(\eta)}\\
&=(u_{t}f)*g(\eta).\end{split}\]\end{proof}

\subsection{An equivariant odd bimodule}
The generator of the one parameter group described in proposition \ref{one} is closely related to the cocycle $c$. On the level of $C_{c}(\mathcal{G})$, pointwise multiplication by $c$ induces a derivation  \cite{Re}, which we will
further investigate in this section.
\begin{proposition}\label{der} Let $\mathcal{G}$ be a locally compact Hausdorff groupoid with Haar system, $c:\mathcal{G}\rightarrow \R$ a regular cocyle, and $\mathcal{H}=\ker
c$.  The operator
\[\begin{split}D:C_{c}(\mathcal{G})&\rightarrow C_{c}(\mathcal{G})\\
f(\xi)&\mapsto c(\xi)f(\xi),\end{split}\]
is a $C_{c}(\mathcal{H})$-linear derivation of $C_{c}(\mathcal{G})$ considered as a bimodule over itself. 
Moreover, it extends to a selfadjoint regular operator in the $C^{*}$-modules $\mathpzc{E}^{\mathcal{G}}\leftrightharpoons C^{*}(\mathcal{H})$ and $\mathpzc{E}^{\mathcal{G}}_{r}\leftrightharpoons C^{*}_{r}(\mathcal{H}).$
\end{proposition}
\begin{proof} It is clear that $D$ is $C_{c}(\mathcal{H})$-linear and the following computation
\[\begin{split}f*Dg(\eta)&=\int_{\mathcal{G}}f(\xi)Dg(\xi^{-1}\eta)d\nu^{r(\eta)}\\
&=\int_{\mathcal{G}}f(\xi)c(\xi^{-1}\eta)g(\xi^{-1}\eta)d\nu^{r(\eta)}\\
&=c(\eta)\int_{\mathcal{G}}f(\xi)g(\xi^{-1}\eta)d\nu^{r(\eta)}-\int_{\mathcal{G}}c(\xi)f(\xi)g(\xi^{-1}\eta)d\nu^{r(\eta)}\\
&=\mathcal{D}(f*g)(\eta)-(Df)*g(\eta),\end{split}\]
shows it is a derivation. Furthermore, it is straightforward to check that \[\langle Df,g\rangle_{\mathcal{H}}=\langle f,Dg\rangle_{\mathcal{H}},\] using formula \ref{innH}. Thus, $D$ is closable, and we will denote its
closure by $D$ as well. It is regular because on $C_{c}(\mathcal{G})$ we have
\[(1+D^{*}D)f(\xi)=(1+c^{2}(\xi))f(\xi),\]
and this clearly has dense range. The same goes for $D+i$ and $D-i$, restricted to $C_{c}(\mathcal{G})$. Therefore, by lemma \ref{closedrange}, these operators are bijective, and hence the Cayley transform $\mathfrak{c}(D)$
(\ref{Cayley})
is unitary. Then, by corollary \ref{symself}, it follows that $D$ is selfadjoint.\end{proof}
The operator $D$ is of course the generator of the one-parameter group of proposition \ref{one}. \newline

From lemma \ref{quotient} we have the identification 
\begin{equation}\label{ident}\mathcal{G}/\mathcal{H}\xrightarrow{\sim}\{(r(\xi),c(\xi)):\xi\in\mathcal{G}\},\end{equation} and for convenience of notation we identify $\mathcal{G}/\mathcal{H}$ with its image in $\mathcal{G}^{(0)}\times \R$. Using this identification, we see that if $K\subset\mathcal{G}^{(0)}$ is compact, the induced map $c:(K\times\R)\cap\mathcal{G}/\mathcal{H}\rightarrow\R$ is proper. 
It is a key fact in the subsequent proof.
\begin{theorem}\label{KK}Let $\mathcal{G}$ be a locally compact Hausdorff groupoid and $c:\mathcal{G}\rightarrow \R$ an exact cocycle. The operator $D$ from proposition \ref{der}, makes the correspondences
\[C^{*}(\mathcal{G})\rightarrow\mathpzc{E}^{\mathcal{G}}\leftrightharpoons C^{*}(\mathcal{H}),\quad C^{*}_{r}(\mathcal{G})\rightarrow\mathpzc{E}^{\mathcal{G}}_{r}\leftrightharpoons C^{*}_{r}(\mathcal{H}),\]
into odd $\R$-equivariant unbounded bimodules.
\end{theorem}
\begin{proof}  The derivation property implies that the commutators $[D,f]$ are bounded for $f\in C_{c}(\mathcal{G})$. They are given by convolution by $Df.$ So it remains to show that $D$ has $C^{*}(\mathcal{H})$-compact 
resolvent. To this end, let $f,\Phi\in C_{c}(\mathcal{G})$. The operator $f\circ (1+D^{2})^{-1}$ acts as
\[f\circ (1+D^{2})^{-1}\Phi(\eta)=\int_{\mathcal{G}}f(\xi)(1+c^{2}(\xi^{-1}\eta))^{-1}\Phi(\xi^{-1}\eta)d\nu^{r(\eta)}\xi.\]
From \ref{modG} and \ref{left}, we see that the action of \[g\in C_{c}(\mathcal{G}\ltimes_{r}\mathcal{G}/\mathcal{H})\subset C^{*}(\mathcal{G}\ltimes_{r}\mathcal{G}/\mathcal{H})=\K_{C^{*}(\mathcal{H})}(\mathpzc{E}^{\mathcal{G}}),\] 
is given by
\begin{eqnarray}\nonumber g\Psi(\eta)&=&\int_{\mathcal{G}\ltimes_{r}\mathcal{G}/\mathcal{H}}g(\xi_{1},[\xi_{2}])\Psi(\xi^{-1}_{1}\eta)d\nu^{[\eta]}(\xi_{1},[\xi_{2}])\\ &=&
\int_\mathcal{G} g(\xi,[\xi^{-1}\eta])\Psi(\xi^{-1}\eta)d\nu^{r(\eta)}\xi. \label{action}\end{eqnarray}
Thus, if we show that for each $f\in C_{c}(\mathcal{G})$ the function
\[k_{f}(\xi,[\eta]):=(1+c^{2}(\eta))^{-1}f(\xi)\]
is a norm limit of elements in $C_{c}(\mathcal{G}\ltimes_{r}\mathcal{G}/\mathcal{H})$, then we are done. \newline\newline 
Define
\[K_{n}:=(r(\supp f)\times\R)\cap c^{-1}([-n,n])\subset\mathcal{G}/\mathcal{H},\]
such that $c(K_{n})\subset [-n,n].$ Here we identify $\mathcal{G}/\mathcal{H}$ with its image in $\mathcal{G}^{(0)}\times\R$ (cf.\ref{ident}), and we view $c$ as a map $\mathcal{G}/\mathcal{H}\rightarrow\R.$ Then
   \[ \dots\subset\dots\subset K_{n}\subset K_{n+1}\subset\dots\subset\mathcal{G}/\mathcal{H},\] is a filtration of $(r(\supp f)\times\R)\cap\mathcal{G}/\mathcal{H}$ by compact sets, cf. lemma \ref{quotient}. Moreover, we may assume that the image of $c$ is
   not a bounded set in $\R$, and that $K_{n}\neq K_{n+1}$ (if not, just rescale).
Thus, there exist cutoff functions \[e_{n}:\mathcal{G}/\mathcal{H}\rightarrow [0,1],\] with \[e_{n}=1 \quad\textnormal{on } K_{n}, \quad e_{n}=0\quad\textnormal{on } \mathcal{G}/\mathcal{H}\setminus K_{n+1}.\]
Define 
\[k^{n}_{f}(\xi,[\eta]):=e_{n}([\eta])k_{f}(\xi,[\eta]).\]

Recall from definitions \ref{full} and \ref{red} that $\|\cdot\|_{r}\leq\|\cdot\|\leq\|\cdot\|_{I}$, so it suffices to show that $\|k^{n}_{f}-k^{m}_{f}\|_{I}\rightarrow 0$ as $n>m\rightarrow\infty$.
For $n>m$ we can estimate:
\[\begin{split}\|k^{n}_{f}-k^{m}_{f}\|_{\nu}&=\sup_{[\eta]\in\mathcal{G}/\mathcal{H}}\int_{\mathcal{G}\ltimes_{r}\mathcal{G}/\mathcal{H}}|k^{n}_{f}(\xi,[\eta])-k^{m}_{f}(\xi,[\eta])|d\nu^{[\eta]}\\
&=\sup_{[\eta]\in\mathcal{G}/\mathcal{H}}\int_{\mathcal{G}}|k_{f}^{n}(\xi,[\eta])-k^{m}_{f}(\xi,[\eta])|d\nu^{r(\eta)}\\
&=\sup_{[\eta]\in\mathcal{G}/\mathcal{H}}\int_{\mathcal{G}}|(e_{n}-e_{m})(\eta)(1+c^{2}(\eta))^{-1}f(\xi)|d\nu^{r(\eta)}\\
&\leq\frac{1}{1+m^{2}}\sup_{[\eta]\in\mathcal{G}/\mathcal{H}}\int_{\mathcal{G}}|f(\xi)|d\nu^{r(\eta)}\\
&=\frac{1}{1+m^{2}}\|f\|_{\nu}.
\end{split}\]
For $\|k^{n}_{f}-k^{m}_{f}\|_{\nu^{-1}}$ a similar computation yields the estimate
\[\|k^{n}_{f}-k^{m}_{f}\|_{I}\leq \frac{1}{1+m^{2}}\|f\|_{I},\] proving that the sequence $k_{f}^{n}$ is Cauchy for $\|\cdot\|_{I}$ and hence for $\|\cdot\|$ and $\|\cdot\|_{r}$. Furthermore, it converges to $f(1+D^{2})^{-1}.$ Since $D$ is the generator of the $\R$-action on $\mathpzc{E}^{\mathcal{G}}$, they commute, and thus the $KK$-cycle is equivariant.\end{proof}

A very simple application of  theorem \ref{KK} recovers the canonical spectral triple on the real line. Consider $\R$ as a groupoid, and take $c=\id:\R\rightarrow\R$. The kernel of $c$ is a point, so $C^{*}(\mathcal{H})=\C$. The spectral triple so obtained is the Fourier transform of the 
canonical Dirac triple $(C_{0}(\R),L^{2}(\R),i\frac{\partial}{\partial x})$ on the line. The canonical triple on the circle (the one point compactification of the line), is obtained directly from the embedding $\Z\rightarrow \R$.

\subsection{Continuous quasi-invariant measures}\label{measures}
An interesting class of cocyles $c:\mathcal{G}\rightarrow\R$ comes from certain well-behaved measures on the unit space $\mathcal{G}^{(0)}$. For this class of cocycles, the kernel algebra $C^{*}(\mathcal{H})$ carries a
canonical trace $\tau:C^{*}(\mathcal{H})\rightarrow\C$. Note that, for an arbitrary $\R$-algebra $A$, $KK_{*}^{\R}(\C,A)\cong K_{*}(A)$, in view of the Baum-Connes conjecture for $\R$, and Connes' Thom isomorphism \cite{Conthom}.  Composition of the induced homomorphism $\tau_{*}:K_{0}(C^{*}(\mathcal{H}))\rightarrow\C$ with the homomorphism $K_{1}(C^{*}(\mathcal{G}))\rightarrow
K_{0}(C^{*}(\mathcal{H}))$ induced by the bimodule coming from $c$, yields an index map $K_{1}(C^{*}(\mathcal{G}))\rightarrow\C$.
\begin{definition} Let $\mathcal{G}$ be a groupoid with Haar system $\{\nu^{x}\}$ and $\mu$ be a positive Radon measure on $\mathcal{G}^{(0)}$. $\nu^{\mu}$ denotes a measure on $\mathcal{G}$, the measure \emph{induced by}
$\mu$, and is defined by
\[\int_{\mathcal{G}} f (\xi)
d\nu^{\mu} (\xi):=\int_{\mathcal{G}^{(0)}}\int_{\mathcal{G}}f(\xi)
d\nu^{x}(\xi) d\mu(x).\]
The measure $\mu$ is said to be \emph{quasi-invariant} if $\nu^{\mu}$ is equivalent to its inverse $\nu_{\mu}$, induced by the corresponding right Haar system on $\mathcal{G}$. The function
\[\Delta:=\frac{d\nu^{\mu}}{d\nu_{\mu}}:\mathcal{G}\rightarrow \R_{+}^{\times},\]
is called the \emph{modular function} of $\mu$. If this function is continuous, then $\mu$ is said to be \emph{continuous}.
\end{definition}
The modular function is an almost everywhere homomorphism \cite{Re}. That is, it is a \emph{measurable} cocycle on $\mathcal{G}$. We will only be interested in continuous measures, and in that case Renault's result is
rephrased as follows.
\begin{proposition} Let $\mathcal{G}$ be a groupoid with Haar system and $\mu$ a continuous quasi-invariant measure on $\mathcal{G}^{(0)}$. Then the modular function $\Delta:\mathcal{G}\rightarrow \R_{+}^{\times}$ is a continuous
cocycle.
\end{proposition}
The measure $\mu$ defines a positive functional $\tau$ on the algebra $C_{c}(\mathcal{G})$. 
\[\begin{split}\tau:C_{c}(\mathcal{G})&\rightarrow \C\\
f&\mapsto \int_{\mathcal{G}^{(0)}}f d\mu.\end{split}\]
It extends to both $C^{*}(\mathcal{G})$ and $C^{*}_{r}(\mathcal{G})$, but in general does not yield a trace. If $\mu$ is quasi-invariant, $\Delta(\xi)\neq 0$ for all $\xi$ in $\mathcal{G}$. Hence we can compose it with the logarithm $\ln :\R_{+}\xrightarrow{\sim} \R$, to obtain a real valued cocycle
$c_{\mu}\in Z^{1}(\mathcal{G},\R)$. We will refer to this element as the \emph{Radon-Nikodym cocycle} on $\mathcal{G}$. If the measure $\mu$ is continuous, the Radon-Nikodym cocycle $c_{\mu}$ induces a one-parameter group $u_{t}$ of automorphisms of $C^{*}(\mathcal{G})$, as mentioned before proposition \ref{one}. Given a one-parameter group $u_{t}$ of automorphisms of a $C^{*}$-algebra $A$, the set of \emph{analytic elements for } $u_{t}$ consists of those $a\in A$ for which the map $t\mapsto u_{t}(a)$ extends to an entire function $\C\rightarrow \C$. It is a dense *-subalgebra of $A$ (see, for example \cite{Ped}, section 8.12.).
\begin{definition}Let $A$ be a $C^{*}$-algebra and $u_{t}$ a strongly continuous one parameter group of automorphisms of $A$. A KMS- $\beta$-\emph{state} on $A$, relative to $u_{t}$, is a state $\sigma:A\rightarrow\C$, such that the for all analytic elements $a,b$  of $A$ the function
\[F:t\mapsto \sigma(au_{t}b)\] admits a continuous bounded continuation to the strip $\{z\in\C:0\leq\im z\leq\beta\}$ that is homolomorphic on the interior, such that
\[F(t+i\beta)=\sigma(u_{t}(b)a).\]
\end{definition}
We refer to \cite{Ped} for a detailed discussion of KMS-states.
\begin{theorem}[\cite{Re}]\label{KMS} Let $\mu$ be a continuous quasi-invariant measure on $\mathcal{G}$. The functional $\tau$ is a KMS $-1$-state for the one parameter group of automorphisms associated to the Radon-Nikodym cocycle on $\mathcal{G}$.
\end{theorem}
Recent work by Exel \cite{Exel} and Kumjian and Renault \cite{Kumre} considers the construction of KMS-states for one parameter groups coming from cocycles. In  \cite{Kumre} it is shown that in case $\mathcal{H}=\ker c$ is principal, every KMS-state for the $u_{t}$ defined by $c$ comes from a quasi-invariant probability measure on $\mathcal{G}^{(0)}$.\newline\newline

A measured groupoid is called \emph{unimodular} if $\Delta=1$ $\nu^{\mu}$-almost everywhere. For continuous measures, the following proposition is a corollary of theorem \ref{KMS}, but it holds for general measures.
\begin{proposition}\label{uni} Let $\mathcal{G}$ be a unimodular measured groupoid. Then the functional $\tau:C_{c}(\mathcal{G})\rightarrow\C$ is a trace.
\end{proposition}
\begin{proof}Compute
\[\begin{split}\tau(f*g)
&=\int_{\mathcal{G}^{(0)}}\int_{\mathcal{G}}f(\xi)g(\xi^{-1}x)d\nu^{x}\xi d\mu(x)\\ &=\int_{\mathcal{G}^{(0)}}\int_{\mathcal{G}}f(\xi)g(\xi^{-1})d\nu^{x}\xi d\mu(x)\\
&=\int_{\mathcal{G}^{(0)}}\int_{\mathcal{G}}f(\xi^{-1})g(\xi)d\nu^{x}\xi d\mu(x)\\&=\tau(g*f),\end{split}\] where we used unimodularity of $\mathcal{G}$ in the third line.\end{proof}
\begin{corollary}\label{measuretrace}Let $\mathcal{G}$ be a groupoid with Haar system, $\mu$ a continuous quasi-invariant measure such that the cocycle $\Delta$ is regular. Then $\tau:C^{*}(\mathcal{H})\rightarrow\C$ is a trace, for
$\mathcal{H}=\ker\Delta$.
\end{corollary}

\begin{corollary}\label{indextrace}Let $\mathcal{G}$ be a groupoid with Haar system, $\mu$ a
continuous quasi-invariant measure such that the cocycle $\Delta$ is exact.
Then $\mu$ induces an index homomorphism
$\Ind_{\mu}:K_{1}(C^{*}(\mathcal{G}))\rightarrow\C$.
\end{corollary}
\begin{proof} By theorem \ref{KK}, the Radon-Nikodym cocycle $c_{\mu}$ defines an element $[D]\in
KK_{1}^{\R}(C^{*}(\mathcal{G}),C^{*}(\mathcal{H}))$ and the Kasparov product with
$[D]$ gives a group homomorphism
$\otimes_{[D]}:K_{1}(C^{*}(\mathcal{G}))\rightarrow K_{0}(C^{*}(\mathcal{H}))$.
The trace $\tau$ induces a homomorphism
$\tau_{*}:K_{0}(C^{*}(\mathcal{H}))\rightarrow\C$. Hence we can define
$\Ind_{\mu}:=\tau_{*}\circ\otimes_{[D]}:K_{1}(C^{*}(\mathcal{G}))\rightarrow
\C$.\end{proof}
Note that, in fact, we get an index map $K_{1}(C^{*}(\mathcal{G}))\rightarrow \C$ for any exact cocycle whose kernel is unimodular with respect to some quasi invariant measure.
\subsection{Integral cocycles on \'{e}tale groupoids}
In this section we focus on cocycles $c:\mathcal{G}\rightarrow \Z$ and relate theorem \ref{KK} to the construction of equivariant $KK$-cycles coming from circle actions given in \cite{CNNR}. First observe that the 1-parameter group (\ref{1par}) gives rise not just to an $\R$-action, but to an action of the circle $\mathbb{T}$. Moreover, we have the following convenient lemma.
\begin{lemma}\label{etalex} Let $\mathcal{G}$ be an \'{e}tale groupoid and $c:\mathcal{G}\rightarrow \Z$ a continuous cocycle. Then $c$ is exact.
\end{lemma}
\begin{proof}By \cite{Re}, proposition I.3.16, $\ker c$ is a locally compact groupoid with Haar system, so $c$ is regular. Since $r\times c$ is continuous, we have that $(r\times c)^{-1}(U)$ is open whenever $U$ is open in $\mathcal{G}^{(0)}\times\R$. To show that $c$ is exact (definition \ref{exact}) let \[U\subset (r\times c)(\mathcal{G})\subset\mathcal{G}^{(0)}\times\R,\] be such that  $(r\times c)^{-1}(U)$ is open in $\mathcal{G}$. Any $U\subset (r\times c)(\mathcal{G})$ is a disjoint union of sets $U_{n}$, $n\in\Z$, with \[U_{n}:=U\cap (\mathcal{G}^{(0)}\times\{n\}),\] because $\Z$ is discrete. Moreover $c^{-1}(n)$ is open, so
\[(r\times c)^{-1}(U_{n})=(r\times c)^{-1}(U)\cap(r\times c)^{-1}(\mathcal{G}^{(0)}\times\{n\}),\]
is open if and only if $(r\times c)^{-1}(U)$ is open. So it suffices to show that $U_{n}$ is open (in the relative topology) whenever $(r\times c)^{-1}(U_{n})$ is open. We have
\[(r\times c)^{-1}(U_{n})=\{\xi\in\mathcal{G}:(r(\xi),c(\xi))=(r(\xi),n)\in U_{n}\},\] 
and $r:\mathcal{G}\rightarrow\mathcal{G}^{(0)}$ is an open map, so we are done.
\end{proof}
This lemma provides us with a myriad of examples, e.g. the class of algebras studied by Renault in \cite{Cuntzlike}. These algebras are constructed from semi direct product groupoids $X\rtimes\sigma$, associated to a local homeomorphism $\sigma$, defined on an open subset $\mathfrak{Dom}\sigma$ of a topological space $X$, onto an open subset of $X$.  The semidirect product groupoid $X\rtimes\sigma$ is defined as
\[\{(x,n,y)\in X\times\Z\times X: \exists k\in\N\quad x\in\mathfrak{Dom}\sigma^{k+n}(x),y\in\mathfrak{Dom}\sigma^{k}(y), \sigma^{k+n}(x)=\sigma^{k}(y)\},\]
with groupoid operations coming from the principal groupoid $X\times X$ and addition in $\Z$. It
can be equipped with a topology that makes it \'{e}tale, in such a way that the map $c:(x,n,y)\mapsto n$ is a continuous cocycle
 $X\rtimes\sigma\rightarrow \Z$. See \cite{Cuntzlike} for more details on this construction. These groupoids hence fit into our framework by lemma \ref{etalex}, and the $KK$-cycle constructed in theorem \ref{KK} is defined for these algebras. This class properly includes all graph $C^{*}$-algebras, and hence all Cuntz- and Cuntz-Krieger algebras (\cite{CuKr, Cuntzlike, Re}). The latter ones were the main source of examples in  \cite{CNNR}, and we will now link our construction to the results in that paper.\newline

In \cite{CNNR}, a circle action $u_{t}$ (i.e. a periodic $\R$-action) on a $C^{*}$-algebra $A$ is the starting point for the construction of an equivariant $KK$-cycle, which is basically given by the generator of this action. Let
\[F:=\{a\in A: u_{t}(a)=a\},\] be the fixed point algebra, and let 
\[A_{k}:=\{a\in A: \forall t\in \R:\quad u_{t}(a)=e^{ikt}a\},\] denote the eigenspaces of the action. In order for the generator to have compact resolvent, the action has to satisfy the following property:
\begin{definition}[\cite{CNNR}] A strongly continuous action of $S^{1}$ on a $C^{*}$-algebra $A$ satisfies the \emph{spectral subspace assumption} (SSA) if the subspaces $F_{k}=\overline{A_{k}A^{*}_{k}}$ are complemented ideals in the fixed point algebra $F$.
\end{definition}
One defines a faithful conditional expectation $\rho:A\rightarrow F$ by
\[\rho(a):=\frac{1}{2\pi}\int_{0}^{2\pi}u_{t}(a)dt,\]
and uses this to define an $F$-valued inner product on by setting $\langle a,b\rangle:=\rho(a^{*}b)$. The completion of $A$ with respect to this inner product is denoted $\mathpzc{E}_{\rho}$. It carries an obvious left $A$ action by adjointable operators. Subsequenlty define projection operators $\rho_{k}:\mathpzc{E}_{\rho}\rightarrow \mathpzc{E}_{\rho},$
\[e\mapsto \frac{1}{2\pi}\int_{0}^{2\pi}e^{-ikt}u_{t}(e)dt.\]
The following equivalence holds:
\begin{lemma}[ \cite{CNNR}]\label{cpt} The circle action on $A$ satisfies the SSA if and only if for all $a\in A$ the operator $a\rho_{k}$ is compact.
\end{lemma}
From this it is easily proved that the generator $D$ of the one parameter group $u_{t}$ has compact resolvent in $\mathpzc{E}_{F}$. We now show that integral cocycles give rise to one parameter groups satisfying the SSA.
\begin{proposition} Let $\mathcal{G}$ be a locally compact \'{e}tale groupoid and $c:\mathcal{G}\rightarrow\Z$ be a cocycle. Then the one parameter group (\ref{1par}) generated by $c$ satisfies the SSA.
\end{proposition}
\begin{proof} Since $\mathcal{G}$ is \'{e}tale, the fixed point algebra $F=C^{*}(\mathcal{H})$ with $\mathcal{H}=\ker c$. The projection operators $\rho_{n}$ correspond to the restrictions $C_{c}(\mathcal{G})\rightarrow C_{c}(c^{-1}(n))$ induced by the inclusions of the closed subspaces $c^{-1}(n)\rightarrow\mathcal{G}$. Moreover $\mathcal{G}$ is a disjoint union
\[\mathcal{G}=\bigcup_{n\in\Z}c^{-1}(n).\] 
Fix $f\in C_{c}(\mathcal{G})$. The space $r(\supp f)\times \Z\cap\mathcal{G}/\mathcal{H}$ (again using \ref{ident}) is a disjoint union of the subsets 
\[K_{n}:=r(\supp f)\times \Z\cap [c^{-1}(n)].\] 
Since $c$ is exact (lemma \ref{etalex}) lemma \ref{quotient} implies each $K_{n}$ is a compact set.
Let $f\in C_{c}(\mathcal{G})$ and $e_{n}\in C_{c}(\mathcal{G}/\mathcal{H})$ be the function that is $1$ on $K_{n}$ and $0$ elsewhere. Define $f_{n}\in C_{c}(\mathcal{G}\rtimes\mathcal{G}/\mathcal{H})$ by $f_{n} (\xi_{1},[\xi_{2}]):=f(\xi_{1})e_{n}([\xi_{2}])$. From equation \ref{action} it follows that the function $f_{k}$ acts as
\[\begin{split}f_{n}\Psi(\eta)&=\int_{\mathcal{G}}f(\xi)e_{n}([\xi^{-1}\eta])\Psi(\xi^{-1}\eta)d\nu^{r(\eta)}\xi\\
&=\int_{\mathcal{G}}f(\xi)\rho_{n}\Psi(\xi^{-1}\eta)d\nu^{r(\eta)}. \end{split}\]
Thus, $f\rho_{n}$ is a compact operator, and by lemma \ref{cpt}, the SSA is satisfied.
\end{proof}

So, in this setting we are dealing with an $\mathbb{T}$-action, and 
the $\mathbb{T}$-equivariant unbounded bimodule induces and index map
\[K_{1}^{\mathbb{T}}(C^{*}(\mathcal{G}))\rightarrow K_{0}^{\mathbb{T}}(C^{*}(\mathcal{H})).\] 
In case a KMS-state for the action is given, e.g. if the cocycle comes from a continuous quasi-invariant measure, 
the restriction of this KMS-state to the fixed point algebra $C^{*}(\mathcal{H})$ gives a trace $\tau:C^{*}(\mathcal{H})\rightarrow\C$, which induces a map 
\[\tau_{*}:K_{0}^{\mathbb{T}}(C^{*}(\mathcal{H}))=K_{0}(C^{*}(\mathcal{H}))[z,z^{-1}]\rightarrow\C[z,z^{-1}] .\]
The resulting index map $K_{1}^{\mathbb{T}}(C^{*}(\mathcal{G}))\rightarrow \C[z,z^{-1}] $ is studied extensively in \cite{CNNR}. This is done by considering the mapping cone $M$ of the inclusion $F\rightarrow A$,
\[M:=\{f\in C_{0}([0,\infty),A):f(0)\in F\}.\]
It comes with an inclusion $\iota:C_{0}(\R,A)\rightarrow M$ as a $C^{*}$-subalgebra by identifying $\R$ with $(0,\infty)$.  An explicit unbounded bimodule $(\hat{\mathpzc{E}},\hat{D})$ for $(M,F)$ is constructed, giving a commutative diagram
\begin{diagram} K_{1}^{\mathbb{T}}(A) &\rTo^{\Ind_{D}} & K_{0}^{\mathbb{T}}(F)\\
\dTo^{\iota_{*}}&\ruTo_{\Ind_{\hat{D}}}&\\
% &  &\\
K_{0}^{\mathbb{T}}(M) &  & \end{diagram}
The map \[K_{0}^{\mathbb{T}}(M)\xrightarrow{\Ind_{D}}K_{0}^{\mathbb{T}}(F)\xrightarrow{\tau_{*}}\C[z,z^{-1}]\] is identified as a kind of equivariant spectral flow (\cite{CNNR}, theorem 1.1). Via $\iota_{*}$ this corresponds to the index map defined above.
\subsection{The noncommutative torus}
Recall that the noncommutative 2-torus, topologically is the $C^{*}$-algebra of an irrational rotation action on the circle $S^{1}$. More precisely, for $\theta\in(0,1)$ consider the action of $\Z$ on $S^{1}$, given by rotation over an angle $2\pi\theta$:
\[e^{2\pi i t} \cdot n:=e^{2\pi i(t+n\theta)}.\]
We denote the corresponding crossed product groupoid by $S^{1}\rtimes_{\theta}\Z$. Lebesgue measure $\lambda$ is quasi-invariant for this action, and we get a representation of $C^{*}(S^{1}\rtimes_{\theta}\Z)$ as bounded operators on the Hilbert space $\mathpzc{H}:=L^{2}(S^{1}\rtimes_{\theta}\Z,\nu^{\lambda})$, where the Haar system $\nu$ is given by counting measures on the fiber. The subalgebra $C_{c}^{\infty}(S^{1}\times_{\theta}\Z)$ comes equipped with two canonical derivations:
\[\partial_{1}f(x,n):= in f(x,n),\quad \partial_{2}f(x,n):=\partial f(x,n).\]
Here $\partial f(x,n)$ is the function $\frac{\partial}{\partial x}f(x,n)$, the differentiation operator on $S^{1}$.
The operator
\[D:=\begin{pmatrix} 0 & -i\partial_{1}-\partial_{2}\\-i\partial_{1}+\partial_{2} &0\end{pmatrix},\]
is an odd, unbounded operator on $\mathpzc{H}_{\theta}:=\mathpzc{H}\oplus\mathpzc{H}$, with compact resolvent. Moreover, $C^{*}(S^{1}\rtimes_{\theta}\Z)$ acts on this graded Hilbert space by the diagonal representation. The commutators $[D,f]$ are bounded, for $f\in C_{c}^{\infty}(S^{1}\times_{\theta}\Z)$, which is dense in $C^{*}(S^{1}\rtimes_{\theta}\Z)$. The above described structure is the \emph{canonical spectral triple on} $C^{*}(S^{1}\rtimes_{\theta}\Z)$.

The trivial cocycle $c:S^{1}\rtimes_{\theta}\Z\rightarrow\Z$, given by projection on the first factor, gives us an unbounded bimodule $(\mathpzc{E}_{\theta},D_{c})$ via theorem \ref{KK}. As a linear space this is just $\ell^{2}(\Z)\tildeotimes C(S^{1})$, with right $C(S^{1})$ action defined by
\[(e_{n}\otimes f)g:=e_{n}\otimes f(g\cdot n),\] and the operator $D_{c}$ acts as $e_{n}\mapsto ne_{n}$, where $e_{n}$ are the canonical basis vectors for $\ell^{2}(\Z)$. Now consider the canonical spectral triple $(L^{2}(S^{1}), i\partial)$ on the circle algebra $C(S^{1})$. This triple is odd and its operator is given by ordinary differentiation. In the proof of the following proposition we will use  \begin{theorem}[\cite{Kuc}]\label{Kuc} Let
$(\mathpzc{E}\tildeotimes_{B}\mathpzc{F},D)\in\Psi_{0}(A,C)$. Supppose that
$(\mathpzc{E},D_{1})\in\Psi_{0}(A,B)$ and $(\mathpzc{F},D_{2})\in\Psi_{0}(B,C)$ are such
that
\begin{enumerate}\item For $x$ in some dense subset of $A\mathpzc{E}$, the operator
\[\left[\begin{pmatrix}D & 0\\ 0& D_{2}\end{pmatrix},\begin{pmatrix} 0 &
T_{x}\\T^{*}_{x} & 0\end{pmatrix}\right]\]
is bounded on $\Dom(D\oplus D_{2})$;
\item $\Dom D\subset \Dom D_{1}\tildeotimes 1$ ;
\item $\langle D_{1}x,Dx\rangle + \langle Dx,D_{1}x\rangle \geq\kappa\langle x,x\rangle$
for all $x$ in the domain of $D$.\end{enumerate} Then
$(\mathpzc{E}\tildeotimes_{B}\mathpzc{F},D)\in\Psi_{0}(A,C)$ represents the
Kasparov product of $(\mathpzc{E},D_{1})\in\Psi_{0}(A,B)$ and $(\mathpzc{F},D_{2})\in\Psi_{0}(B,C)$.\end{theorem}
In order to be able to do so, we have to turn the given odd modules into equivalent even ones. The standard procedure to do so is to associate to $(\mathpzc{E},D)$ the graded module $\mathpzc{E}\oplus\mathpzc{E}$ with the operator $\begin{pmatrix} 0 & D\\ D& 0\end{pmatrix}$. This module comes with a right $\C_{1}$-action, implemented by the matrix $\begin{pmatrix} 0& i\\ i& 0\end{pmatrix}$. In our case this yields bimodules for the pairs $(C^{*}(S^{1}\rtimes \Z),C(S^{1})\otimes\C_{1})$ and $(C(S^{1}),\C_{1})$. Tensoring the last bimodule once again by $\C_{1}$, we get a bimodule for $(C(S^{1})\otimes\C_{1},\C_{2})$. The Kasparov product of these modules is an element of $KK_{0}(C^{*}(S^{2}\rtimes \Z, \C_{2})$, which by formal Bott periodicity is isomorphic to $KK_{0}(C^{*}(S^{2}\rtimes \Z, \C)=K^{0}(C^{*}(S^{2}\rtimes \Z))$.
\begin{proposition} The class $[(\mathpzc{H}_{\theta}, D)]\in K^{0}(C^{*}(S^{1}\rtimes_{\theta}\Z))$ is the Kasparov product of $[(\mathpzc{E}_{\theta},D_{c})]\in KK_{1}(C^{*}(S^{1}\rtimes_{\theta}\Z),C(S^{1}))$ with $[(L^{2}(S^{1}), i\partial)]\in K^{1}(C(S^{1}))$.
\end{proposition}
\begin{proof} It is immediate that $\mathpzc{E}_{\theta}\tildeotimes_{C(S^{1})}L^{2}(S^{1})\cong\mathpzc{H}$. Thus, the graded $C^{*}$-module tensor product \[(\mathpzc{E}_{\theta}\oplus\mathpzc{E}_{\theta})\otimes \C_{1}\tildeotimes_{C(S^{1})\otimes\C_{1}}(L^{2}(S^{1})\oplus L^{2}(S^{1}))\otimes\C_{2})\cong \mathpzc{H}_{\theta}\oplus\mathpzc{H}_{\theta}^{\textnormal{op}},\]
with the standard odd grading, and $\C_{2}$-action implemented by the operators $\begin{pmatrix} 0& i\\i&0\end{pmatrix}$ and $\begin{pmatrix} 0& 1\\-1&0\end{pmatrix}$. We show that $D\oplus D^{\textnormal{op}}$ (which we will abusively denote by $D$ from now on) satisfies the conditions of Kucerovsky's theorem, with respect to \[D_{1}=\begin{pmatrix} 0 & D_{c}\\ D_{c}& 0\end{pmatrix},\quad\textnormal{ and}\quad  D_{2}=\begin{pmatrix} 0 & i\partial\\ i\partial& 0\end{pmatrix}\otimes 1.\] To prove (1), let $x\in C_{c}^{\infty}(S^{1}\rtimes\Z)\oplus C_{c}^{\infty}(S^{1}\rtimes\Z)$, which is dense in $\mathpzc{E}_{\theta}\oplus\mathpzc{E}_{\theta}$. It suffices to consider homogenous elements $x$ with support in $S^{1}\times\{n\}$ for $n\in\Z$. A tedious but straightforward computation yields that 
\[\left[\begin{pmatrix}D & 0\\ 0& T\end{pmatrix},\begin{pmatrix} 0 &
T_{x}\\T^{*}_{x} & 0\end{pmatrix}\right]\begin{pmatrix} e\otimes f\\ g\end{pmatrix}=\begin{pmatrix} D_{1}x\otimes g+(-1)^{\partial x}i\partial(x\cdot(-n))g \\ \langle i \partial(x\cdot(-n)),e\rangle f-(-1)^{\partial x}\langle D_{1} x,e\rangle f\end{pmatrix}.\]
We tacitly identify $x$ with the function on $S^{1}\times \{n\}$, an object we can differentiate, and therefore this expression is bounded, and extends to finite sums, i.e. to $C_{c}^{\infty}(S^{1}\times \Z)$. Condition $(2)$ is obvious, while $D_{1}\otimes 1$ only acts on the $\Z$ part of functions, whereas $D$ is essentially $D_{1}\otimes 1$ plus differentiation in the $S^{1}$ direction. Another tedious calculation (using that the derivations $\partial _{1}$ and $\partial _{2}$ commute on a common core) shows that 
\[\langle D_{1}x,Dx\rangle + \langle Dx,D_{1}x\rangle\geq \langle D_{1}x,D_{1}x\rangle + \langle D_{1}x,D_{1}x\rangle\geq 0, \]
for all $x$ in $\bigoplus_{i=1}^{4} C_{c}^{\infty}(S^{1}\rtimes\Z)$, which is a common core for $D$ and $D_{1}$.
\end{proof}

\end{document}